\documentclass[leqno,11pt,a4paper]{amsart}
\usepackage{graphicx}
\usepackage[usenames,dvipsnames]{color}
\usepackage{hyperref}
\usepackage{mathabx}
\usepackage{algpseudocode}
\usepackage{mleftright}
\usepackage{tabularx}
\usepackage{booktabs}
\usepackage{colortbl}
\usepackage{array}
\newtheorem{thm}{Theorem}[section]
\newtheorem{lemma}[thm]{Lemma}
\newtheorem{cor}[thm]{Corollary}
\newtheorem{prop}[thm]{Proposition}
\theoremstyle{definition}

\newtheorem{remark}[thm]{Remark}

\newtheorem{definition}[thm]{Definition}

\numberwithin{equation}{section}
\long\def\blankfootnotetext#1{\begingroup\def\thefootnote{\fnsymbol{footnote}}\footnotetext{#1}\endgroup}
\newcommand{\orig}{\mathbf{0}}
\newcommand{\Z}{\mathbb{Z}}
\newcommand{\Q}{\mathbb{Q}}
\newcommand{\C}{\mathbb{C}}
\newcommand{\NQ}{N_\Q}
\newcommand{\GL}{GL}
\newcommand{\Proj}{\mathbb{P}}
\newcommand{\Hom}[1]{\mathrm{Hom}\mleft({#1}\mright)}
\newcommand{\Vol}[1]{\mathrm{Vol}\mleft({#1}\mright)}
\newcommand{\abs}[1]{\left\vert{#1}\right\vert}
\newcommand{\mult}[1]{\mathrm{mult}\mleft({#1}\mright)}
\newcommand{\intr}[1]{{#1}^\circ}
\newcommand{\V}[1]{\mathrm{vert}\mleft({#1}\mright)}
\newcommand{\conv}[1]{\mathrm{conv}\mleft({#1}\mright)}
\newcommand{\sconv}[1]{\mathrm{conv}\mleft\{{#1}\mright\}}

\newcommand{\Ehr}{\mathrm{Ehr}}
\newcommand{\lcm}[1]{\mathrm{lcm}\mleft\{{#1}\mright\}}
\renewcommand{\gcd}[1]{\mathrm{gcd}\mleft\{{#1}\mright\}}

\renewcommand{\max}[1]{\mathrm{max}\mleft\{{#1}\mright\}}
\renewcommand{\det}[1]{\mathrm{det}\mleft({#1}\mright)}

\newcommand{\intpart}[1]{\left\lfloor{#1}\right\rfloor}
\newcommand{\fracpart}[1]{\left\{{#1}\right\}}
\newcommand{\modpart}[1]{\fracpart{\lambda_{#1}\kappa/h}}
\newcommand{\dmodpart}[1]{\fracpart{\frac{\lambda_{#1}\kappa}{h}}}
\newcommand{\invmodpart}[1]{\fracpart{\lambda_{#1}(h-\kappa)/h}}
\newcommand{\dinvmodpart}[1]{\fracpart{\frac{\lambda_{#1}(h-\kappa)}{h}}}
\newcommand{\wholepart}[1]{\intpart{\lambda_{#1}\kappa/h}}

\newcommand{\oddrow}{\rowcolor[gray]{0.95}}
\newcommand{\evnrow}{}
\newcolumntype{g}{>{\columncolor[gray]{0.95}\centering\arraybackslash}m{2.8em}}
\newcolumntype{w}{>{\centering\arraybackslash}m{2.8em}}
\newcommand{\ForTo}{\textbf{to}\ }

\newcommand{\AndAlso}{\textbf{and}\ }
\begin{document}
\author[A.~M.~Kasprzyk]{Alexander M.~Kasprzyk}
\address{Department of Mathematics\\Imperial College London\\180 Queen's Gate\\London SW7 2AZ\\UK}
\email{a.m.kasprzyk@imperial.ac.uk}
\blankfootnotetext{2010 \emph{Mathematics Subject Classification}: 14M25 (Primary); 14J45, 52B20 (Secondary).}
\blankfootnotetext{The author is supported by EPSRC grant EP/I008128/1.}
\title{Classifying terminal weighted projective space}
\begin{abstract}
We present a classification of all weighted projective spaces with at worst terminal or canonical singularities in dimension four. As a corollary we also classify all four-dimensional one-point lattice simplices up to equivalence. Finally, we classify the terminal Gorenstein weighted projective spaces up to dimension ten.
\end{abstract}
\maketitle
\section{Introduction}
In this paper we classify weighted projective spaces $X=\Proj(\lambda_0,\lambda_1,\ldots,\lambda_n)$ with terminal or canonical singularities. The reader, however, need have no knowledge of the geometry involved for, as we justify below, we rapidly move to the language of combinatorics and lattice simplices with one interior lattice point. Whichever viewpoint one takes, these are important and natural objects about which surprisingly little is known beyond dimension three. In~\S\ref{sec:combinatorial_results} we prove a number of combinatorial conditions on the weights of $X$; these are used in~\S\ref{sec:dim4} to classify all four-dimensional weighted projective spaces with at worst terminal or canonical singularities. Finally~\S\ref{sec:terminal_reflexive_general} and~\S\ref{sec:terminal_reflexive_computer} are dedicated to classifying terminal Gorenstein weighted projective spaces, which we do up to dimension ten.

\subsection*{Terminal and canonical singularities}
Terminal singularities were introduced by Reid, and are unavoidable in birational geometry~\cite{Reid83M}. They form the smallest class of singularities that must be allowed if one wishes to construct minimal models in dimensions three or more. Canonical singularities can be regarded as the limit of terminal singularities; they arise naturally as the singularities occurring on the canonical models of varieties of general type~\cite{Reid83M,Reid85}.

In the context of toric geometry terminal and canonical singularities have a particularly elegant combinatorial description. We take this opportunity to fix our notation; for details see~\cite{Dan78}. Let $M\cong\Z^n$ be the character lattice of the algebraic torus $(\C^\times)^n$, with dual lattice $N:=\Hom{M,\Z}$. Write $\NQ:=N\otimes_\Z\Q$ for the corresponding rational vector space. A toric singularity corresponds to a strictly convex rational polyhedral cone $\sigma\subset\NQ$. The cone $\sigma$ is terminal if and only if:
\begin{enumerate}
\item\label{item:Q_Gorenstein}
the primitive lattice points $\rho_1,\ldots,\rho_m$ corresponding to the rays of $\sigma$ are contained in an affine hyperplane $H_u:=\{v\in\NQ\mid u(v)=1\}$ for some $u\in M_\Q$;
\item\label{item:terminal_hyper}
with the exception of the origin $\orig$ and the generators $\rho_i$ of the rays, no other lattice points of $N$ are contained in the part of $\sigma$ on or under $H_u$, i.e.
$$N\cap\sigma\cap\{v\in\NQ\mid u(v)\leq 1\}=\{\orig,\rho_1,\ldots,\rho_m\}.$$
\end{enumerate}
The cone $\sigma$ is canonical if and only if~\eqref{item:Q_Gorenstein} holds and
\begin{enumerate}
\item[(ii$'$)]\label{item:canonical_hyper}
the origin $\orig$ is the only lattice point contained in the part of $\sigma$ under $H_u$, i.e.
$$N\cap\sigma\cap\{v\in\NQ\mid u(v)<1\}=\{\orig\}.$$
\end{enumerate}
When the hyperplane $H_u$ in condition~\eqref{item:Q_Gorenstein} corresponds to a lattice point $u\in M$, the singularity is Gorenstein.

\subsection*{Weighted projective space}
Let $(\lambda_0,\lambda_1,\ldots,\lambda_n)\in\Z_{>0}^{n+1}$ and define $S(\lambda_0,\lambda_1,\ldots,\lambda_n)$ to be the polynomial algebra $\C[x_0,x_1,\ldots,x_n]$ graded by $\deg{x_i}=\lambda_i$. The projective variety
$$\Proj(\lambda_0,\lambda_1,\ldots,\lambda_n):=\mathrm{Proj}\left(S(\lambda_0,\lambda_1,\ldots,\lambda_n)\right)$$
is called \emph{weighted projective space}. Since $S(\lambda_0,\lambda_1,\ldots,\lambda_n)\cong S(k\lambda_0,k\lambda_1,\ldots,k\lambda_n)$ for any positive integer $k$, we require from here onwards that the weights $(\lambda_0,\lambda_1,\ldots,\lambda_n)$ are coprime.

Weighted projective space can also be defined in terms of a group action: let $\C^\times$ act on $\C^{n+1}$ via
$$\mu\cdot(x_0,x_1,\ldots,x_n)=(\mu^{\lambda_0}x_0,\mu^{\lambda_1}x_1,\ldots,\mu^{\lambda_n}x_n).$$
Then $\Proj(\lambda_0,\lambda_1,\ldots,\lambda_n)$ is given by the quotient $(\C^{n+1}\setminus\{0\})/\,\C^\times$.  For details see~\cite{Dolg82,I-F00}.

Any toric variety $X$ determines, and is determined by, a fan $\Delta$ in the lattice $N$; see~\cite{Dan78}. Two toric varieties $X_1$ and $X_2$ are isomorphic if and only if the corresponding fans $\Delta_1$ and $\Delta_2$ are isomorphic with respect to some element of $\GL_n(\Z)$. The fan of weighted projective space can be characterised as follows. Let $\{\rho_0,\rho_1,\ldots,\rho_n\}\subset N$ be a set of $n+1$ primitive lattice points such that:
\begin{enumerate}
\item\label{item:wps_bary_zero}
$\lambda_0\rho_0+\lambda_1\rho_1+\ldots+\lambda_n\rho_n=\orig$, where $\lambda_i\in\Z_{>0}$, $\gcd{\lambda_0,\lambda_1,\ldots,\lambda_n}=1$;
\item\label{item:wps_index_1}
the $\rho_i$ generate the lattice $N$.
\end{enumerate}
Then, up to isomorphism, the complete simplicial fan $\Delta$ with rays generated by the $\rho_i$ corresponds to $\Proj(\lambda_0,\lambda_1,\ldots,\lambda_n)$~\cite[Proposition~2]{BB92}. Notice that the convex hull
$$P=\sconv{\rho_0,\rho_1,\ldots,\rho_n}\subset\NQ$$
is a lattice $n$-simplex. By a slight abuse of terminology, we refer to any simplex $P\subset\NQ$ with primitive vertices satisfying~\eqref{item:wps_bary_zero} and~\eqref{item:wps_index_1} as \emph{the} simplex associated with $\Proj(\lambda_0,\lambda_1,\ldots,\lambda_n)$; in other words, we consider simplices only up to $\GL_n(\Z)$-equivalence. The weighted projective space has at worst terminal singularities if and only if $P\cap N=\V{P}\cup\{\orig\}$, and has at worst canonical singularities if and only if the origin is the only interior lattice point of $P$, that is $\intr{P}\cap N=\{\orig\}$.

If we drop condition~\eqref{item:wps_index_1}, allowing the $\rho_i$ to generate a finite index sublattice of $N$, the resulting toric variety is called a \emph{fake weighted projective space}~\cite{Buc08,Kas08b}. In this case $P$ is a Fano simplex: a simplex containing the origin in its strict interior, and with primitive vertices (see~\cite{KN12} for a survey of Fano polytopes). Terminal fake weighted projective spaces are in bijective correspondence with one-point lattice simplices.

We will make one additional assumption on the weights $(\lambda_0,\lambda_1,\ldots,\lambda_n)\in\Z_{>0}^{n+1}$. We assume henceforth that they are \emph{well-formed}, that is, that for each subsequence of length $n$,
$$\gcd{\lambda_0,\ldots,\lambda_{i-1},\widehat{\lambda}_i,\lambda_{i+1},\ldots,\lambda_n}=1.$$
In the context of toric geometry this is not a restriction: since the $\rho_i$ are primitive lattice elements, any coprime weights are necessarily well-formed.

\subsection*{Dimensions three and four}
The three-dimensional weighted projective spaces with at worst terminal singularities were classified in~\cite{Kas03}; the canonical classification was derived in \cite{Kas08a}. There are, respectively, $7$ and $104$ isomorphism classes.

We classify four-dimensional weighted projective spaces in~\S\ref{sec:dim4}. Up to isomorphism, we find $28,\!686$ weights giving at worst terminal singularities, and $338,\!752$ weights giving at worst canonical singularities. The maximum degrees and sum of weights in the two cases are recorded in Theorems~\ref{thm:terminal_dim_4} and~\ref{thm:canonical_dim_4}, respectively. The terminal case is particularly interesting, suggesting a conjectural form for the terminal weighted projective space of largest degree in dimension $\geq 4$ (see Lemma~\ref{lem:maximum_terminal_candidate} and Remark~\ref{rem:comments_on_conjectures}). Finally, in Theorem~\ref{thm:one_point_classification} we give a classification, up to equivalence, of all one-point lattice simplices in dimension four; equivalently, all four-dimensional terminal fake weighted projective spaces.

\subsection*{Gorenstein weighted projective space}
Gorenstein weighted projective spaces give rise to \emph{reflexive simplices}, that is, simplices $P$ whose dual $P^\vee:=\{u\in M_\Q\mid u(v)\geq -1\text{ for all }v\in P\}$ is also a lattice simplex. They have been studied in some detail: the duality of reflexive polytopes gives rise to mirror-symmetric Calabi--Yau varieties as general hypersurfaces in the corresponding toric varieties~\cite{Bat94,KS02}. The possible weights are characterised by unit partitions, that is, sums of the form $1/k_0+\ldots+1/k_n=1$. Using this description, Nill~\cite{Nill04} was able to derive expressions for the maximum degree and maximum sum of weights in terms of the Sylvester sequence. These bounds are sharp in the canonical case.

In~\S\ref{sec:terminal_reflexive_general} we concentrate on Gorenstein weighted projective spaces with at worst terminal singularities; essentially nothing is known about the behaviour of this important subclass of reflexive polytopes. Their classification in dimensions four and five are derived in Theorems~\ref{thm:gorenstein_terminal_dim_4} and~\ref{thm:gorenstein_terminal_dim_5}, respectively. Higher dimensional classifications, produced with the aid of a computer, form the focus of~\S\ref{sec:terminal_reflexive_computer}. In contrast with the canonical case, the number of isomorphism classes and maximum degree grow remarkable slowly, as illustrated by Tables~\ref{tab:terminal_list} and~\ref{tab:max_deg}.

\begin{table}[tb]
\centering
\caption{The number of isomorphism classes of terminal Gorenstein weighted projective space in dimension $n$.}
\label{tab:terminal_list}
\begin{tabular}{rgwgwgwgw}
\toprule
Dimension $n$&$3$&$4$&$5$&$6$&$7$&$8$&$9$&$10$\\
Number of weights&$1$&$2$&$4$&$18$&$135$&$1342$&$21703$&$591540$\\
\bottomrule
\end{tabular}
\end{table}

\section{Terminal and canonical weighted projective space}\label{sec:combinatorial_results}
We begin by establishing conditions on the weights when the corresponding weighted projective space is terminal or canonical. In order to do this, we will require the following notation:

\begin{definition}
Let $q\in\Q$. The \emph{integer part} of $q$ is given by $\intpart{q}:=\max{a\in\Z\mid a\leq q}$, and the \emph{fractional part} of $q$ is given by $\fracpart{q}:=q-\intpart{q}$.
\end{definition}

\begin{definition}
Given weights $(\lambda_0,\lambda_1,\ldots,\lambda_n)\in\Z_{>0}^{n+1}$, we denote the sum by:
$$h:=\lambda_0+\lambda_1+\ldots+\lambda_n.$$
\end{definition}

\begin{prop}\label{prop:terminal_sum}
Let $(\lambda_0,\lambda_1,\ldots,\lambda_n)\in\Z_{>0}^{n+1}$ be well-formed. The weighted projective space $\Proj(\lambda_0,\lambda_1,\ldots,\lambda_n)$ has at worst terminal singularities if and only if $\sum_{i=0}^n\modpart{i}\in\{2,\ldots,n-1\}$ for each $\kappa\in\{2,\ldots,h-2\}$.
\end{prop}
\begin{proof}
This is a generalisation of~\cite[Proposition~1.4]{Kas03}. Let $P\subset\NQ$ be the $n$-simplex associated with $X=\Proj(\lambda_0,\lambda_1,\ldots,\lambda_n)$, so that the spanning fan of $P$ with $n+1$ rays $\Q_{\geq 0}\rho_i$ defines $X$ as a toric variety. By~\cite[Proposition~2]{BB92} we have that $\V{P}$ generates the lattice $N$; moreover since $\lambda_0\rho_0+\lambda_1\rho_1+\ldots+\lambda_n\rho_n=\orig$ we can write any lattice point as a sum of positive multiples of the $\rho_i$. 

Let $\varphi:\Z^{n+1}\rightarrow N$ be the map given by sending the standard basis of $\Z^{n+1}$ to the vertices of $P$, so $\varphi:e_i\mapsto\rho_{i-1}$. We equip $\Z^{n+1}$ with a grading $u=(1,1,\ldots,1)\in\Hom{\Z^{n+1},\Z}$, and refer to $u(v)$ as the \emph{height} of $v$. Let
$$\mathcal{Z}:=\Z^{n+1}+\left<\frac{1}{h}(\lambda_0,\lambda_1,\ldots,\lambda_n)\right>$$
be the overlattice of $\Z^{n+1}$ generated by introducing the rational vector $\frac{1}{h}(\lambda_0,\lambda_1,\ldots,\lambda_n)$, with inclusion map $\iota:\Z^{n+1}\hookrightarrow\mathcal{Z}$. We have that $\varphi\left(\frac{1}{h}(\lambda_0,\lambda_1,\ldots,\lambda_n)\right)=\orig\in N$, and $\varphi$ induces a bijection between the points $\{v\in\mathcal{Z}\mid u(v)=1\}$ at height one and $N$.

Define $\Gamma':=\sconv{\orig,e_1,\ldots,e_{n+1},e_1+\ldots+e_{n+1}}$ to be the parallelepiped in $\Z^{n+1}$ generated by the standard basis elements. By making the appropriate identifications, we regard $\Gamma'$ as a fundamental domain. Set $\Gamma:=\iota(\Gamma')$ to be the corresponding image in $\mathcal{Z}$. The non-vertex lattice points of $P\cap N$ are in bijective correspondence with the height-one points $\{v\in\Gamma\cap\mathcal{Z}\mid u(v)=1\}$. The lattice points in $\Gamma$ not identified with the origin are generated by multiples $\kappa\times\frac{1}{h}(\lambda_0,\lambda_1,\ldots,\lambda_n)$, where $\kappa\in\{1,\ldots,h-1\}$. We shall show that the representatives of these multiples in $\Gamma$ are all non-zero and distinct.

Suppose that there exist $\kappa_1,\kappa_2\in\{0,1,\ldots,h-1\}$, $\kappa_1>\kappa_2$, such that the multiples $\kappa_j\times\frac{1}{h}(\lambda_0,\lambda_1,\ldots,\lambda_n)$, $j=1,2$, identify with the same point in $\Gamma$. Then $\{\lambda_i\kappa_1/h\}=\{\lambda_i\kappa_2/h\}$ for each $i\in\{0,\ldots,n\}$. Hence
$$(\kappa_1-\kappa_2)\frac{\lambda_i}{h}=\intpart{\frac{\lambda_i\kappa_1}{h}}-\intpart{\frac{\lambda_i\kappa_2}{h}}\in\Z_{>0},$$
and so $h\mid(\kappa_1-\kappa_2)\lambda_i$, for each $i$. Since $\gcd{\lambda_0,\lambda_1,\ldots,\lambda_n}=1$ there exist $\mu_i\in\Z$ such that $\mu_0\lambda_0+\mu_1\lambda_1+\ldots+\mu_n\lambda_n=1$. In particular,
$$\sum_{i=0}^n\mu_i(\kappa_1-\kappa_2)\lambda_i=\kappa_1-\kappa_2,$$
and so $h\mid\kappa_1-\kappa_2$. But this implies that $\kappa_1\geq h$, a contradiction.

When $\kappa=1$ we recover the origin in $P$. Notice that $\fracpart{\lambda_i/h}\neq 0$ for any $i\in\{0,\ldots,n\}$, so $\fracpart{\lambda_i(h-1)/h}=1-\fracpart{\lambda_i/h}$, giving:
$$\sum_{i=0}^n\fracpart{\frac{\lambda_i(h-1)}{h}}=(n+1)-\sum_{i=0}^n\fracpart{\frac{\lambda_i}{h}}=n.$$
Hence when $\kappa=h-1$ we have a point in $\Gamma$ at height $n$. Conversely suppose that $\kappa$ is such that the corresponding point in $\Gamma$ is at height $n$. Since $\modpart{i}<1$ it follows that $\modpart{i}\ne 0$ for any $i\in\{0,\ldots,n\}$. We see that
$$\sum_{i=0}^n\fracpart{\frac{\lambda_i(h-\kappa)}{h}}=(n+1)-\sum_{i=0}^n\dmodpart{i}=1.$$
Since none of the remaining points given by $\kappa\in\{2,\ldots,h-2\}$ can correspond to either the origin of $P$ or to a vertex of $P$, the result follows.
\end{proof}

\begin{prop}\label{prop:terminal_coprimeness}
Let $\Proj(\lambda_0,\lambda_1,\ldots,\lambda_n)$ have at worst terminal singularities. Then
$$\gcd{\lambda_0,\ldots,\lambda_{i-1},\widehat{\lambda}_i,\lambda_{i+1},\ldots,\lambda_{j-1},\widehat{\lambda}_j,\lambda_{j+1},\ldots,\lambda_n}=1$$
for all sequences of $n-1$ weights given by omitting $\lambda_i$ and $\lambda_j$, where $0\leq i<j\leq n$.
\end{prop}
\begin{proof}
Set $m=\gcd{\lambda_0,\ldots,\lambda_{i-1},\widehat{\lambda}_i,\lambda_{i+1},\ldots,\lambda_{j-1},\widehat{\lambda}_j,\lambda_{j+1},\ldots,\lambda_n}$. We have:
$$\frac{\lambda_i}{m}\rho_i+\frac{\lambda_j}{m}\rho_j=-\sum_{k\ne i,j}^n\frac{\lambda_k}{m}\rho_k\in N.$$
Since the weighted projective space is terminal the triangle with vertices $\{\orig,\rho_i,\rho_j\}$ is lattice point free, thus there exists an element of $\GL_n(\Z)$ sending $\rho_i\mapsto e_1$ and $\rho_j\mapsto e_2$. Hence it must be that $m\mid\lambda_i$ and $m\mid\lambda_j$. Since the weights are coprime, $m=1$.
\end{proof}

The proof of Proposition~\ref{prop:terminal_sum} generalises to the canonical setting:

\begin{prop}\label{prop:canonical_sum}
Let $(\lambda_0,\lambda_1,\ldots,\lambda_n)\in\Z_{>0}^{n+1}$ be well-formed. The weighted projective space $\Proj(\lambda_0,\lambda_1,\ldots,\lambda_n)$ has at worst canonical singularities if and only if $\sum_{i=0}^n\modpart{i}\in\{1,\ldots,n-1\}$ for each $\kappa\in\{2,\ldots,h-2\}$.
\end{prop}
\begin{proof}
The proof parallels that of the terminal case, with the exception that $P$ may contain non-vertex lattice points on its boundary $\partial P$. These points correspond to those height-one multiples of $\frac{1}{h}(\lambda_0,\lambda_1,\ldots,\lambda_n)$ such that $\modpart{i}=0$ for some $i\in\{0,\ldots,n\}$. We have already seen in the proof of Proposition~\ref{prop:terminal_sum} that if $\kappa$ is such that the corresponding point in $\Gamma$ is at height $n$, then $h-\kappa$ gives a point at height one. Furthermore, $\fracpart{\lambda_i(h-\kappa)/h}=1-\modpart{i}\ne 0$. Conversely, if $\kappa$ corresponds to a point in $\Gamma$ at height one such that $\modpart{i}\ne 0$ for all $i\in\{0,\ldots,n\}$, we see that $h-\kappa$ gives a point at height $n$.
\end{proof}

As a consequence of Proposition~\ref{prop:canonical_sum} we have a combinatorial proof of the following:

\begin{cor}\label{cor:gorenstein_cones_are_canonical_cones}
Let $(\lambda_0,\lambda_1,\ldots,\lambda_n)\in\Z_{>0}^{n+1}$ be well-formed, and let $P\subset\NQ$ be the $n$-simplex associated with the corresponding weighted projective space. If $\lambda_j\mid h$ then
$$\conv{F_j\cup\{\orig\}}\cap N=(F_j\cap N)\cup\{\orig\},$$
where $F_j:=\sconv{\rho_0,\ldots,\rho_{j-1},\widehat{\rho}_j,\rho_{j+1},\ldots,\rho_n}$ is the facet of $P$ not containing $\rho_j$.
\end{cor}
\begin{proof}
Suppose that $v$ is a non-zero, non-vertex lattice point in $P$. The ray from the origin through $v$ intersects the boundary of $P$ at some (possibly rational) point $\bar{v}\in\partial P$. In particular, $\bar{v}$ must lie on some facet $F_j$ of $P$, so that there exist non-negative rational coefficients $\beta_i\in\Q_{\geq 0}$ such that:
$$\bar{v}=\sum_{i=0}^n\beta_i\rho_i,\qquad\text{ where }\sum_{i=0}^n\beta_i=1,\beta_j=0.$$
Let $0<a/b\leq 1$, $\gcd{a,b}=1$, be such that $v=(a/b)\bar{v}+(1-a/b)\orig$. Then:
\begin{equation}\label{eq:bary_v_1}
v=\sum_{i=0}^n\left(\frac{a}{b}\beta_i+\left(1-\frac{a}{b}\right)\frac{\lambda_i}{h}\right)\rho_i,\qquad\text{ where }\sum_{i=0}^n\left(\frac{a}{b}\beta_i+\left(1-\frac{a}{b}\right)\frac{\lambda_i}{h}\right)=1.
\end{equation}
By Proposition~\ref{prop:canonical_sum} we know that there exists $\kappa\in\{2,\ldots,h-2\}$ such that:
\begin{equation}\label{eq:bary_v_2}
v=\sum_{i=0}^n\dmodpart{i}\rho_i,\qquad\text{ where }\sum_{i=0}^n\dmodpart{i}=1.
\end{equation}
By uniqueness of barycentric coordinates, the coefficients in equations~\eqref{eq:bary_v_1} and~\eqref{eq:bary_v_2} are equal. In particular,
$$\dmodpart{j}=\left(1-\frac{a}{b}\right)\frac{\lambda_j}{h}.$$
With a little rearranging we see that $\wholepart{j}bh=\left(b(\kappa-1)+a\right)\lambda_j$, and so $bh\mid\left(b(\kappa-1)+a\right)\lambda_j$.

Now suppose that $\lambda_j\mid h$. Then there exists $m\in\Z_{>0}$ such that $h=m\lambda_j$. In particular, $bm\mid b(\kappa-1)+a$ and so $b\mid a$. But $\gcd{a,b}=1$, so $a=b=1$ and $v=\bar{v}$; i.e.~$v\in F_j$.
\end{proof}

\begin{remark}\label{rem:gorenstein_are_canonical}
Recall that $P$ is reflexive if and only if $\lambda_i\mid h$ for each $i\in\{0,\ldots,n\}$ (see, for example,~\cite[Lemma~3.5.6]{CK-MirrorSymmetry}). An immediate consequence of Corollary~\ref{cor:gorenstein_cones_are_canonical_cones} is that any Gorenstein weighted projective space has at worst canonical singularities (of course this is known to be true more generally for all Gorenstein toric vareities~\cite[Corollary~3.6]{Reid82}).
\end{remark}

\subsection*{Connections with Ehrhart theory}
The proofs of Propositions~\ref{prop:terminal_sum} and~\ref{prop:canonical_sum} have natural interpretations in terms of Ehrhart theory. Let $(\lambda_0,\lambda_1,\ldots,\lambda_n)\in\Z_{>0}^{n+1}$ be well-formed and let $P$ be the $n$-simplex associated with the corresponding weighted projective space. Define:
$$\delta_j:=\abs{\left\{\kappa\in\{0,\ldots,h - 1\}\ \Big|\ \sum_{i=0}^n\dmodpart{i}=j\right\}}.$$
Then $\sum_{j=0}^n\delta_j=h$, and $\delta:=(\delta_0,\delta_1,\ldots,\delta_n)$ is the Ehrhart $\delta$-vector (or $h^*$-vector) of $P$, that is,
$$\Ehr_P(t):=\frac{\delta_0+\delta_1t+\ldots+\delta_nt^n}{(1-t)^{n+1}}$$
is a rational function whose power-series expansion $\Ehr_P(t)=c_0+c_1t+c_2t^2+\ldots$ is a generating function for the number of lattice points in successive dilations $c_m=\abs{mP\cap N}$ of $P$.

For a general lattice polytope $Q\subset\NQ$, the $\delta_i$ are known to satisfy certain conditions; see~\cite[Theorem~3.5]{BLDPS05} for a concise overview. In particular, $\delta_0=1$, $\delta_1=\abs{Q\cap N}-n-1$, $\delta_n=\abs{\intr{Q}\cap N}$, and $\Vol{Q}=\sum_{j=0}^n\delta_j$. In our setting we see that if $P$ is terminal then $\delta=(1,1,\delta_2,\ldots,\delta_{n-1},1)$, and if $P$ is canonical then $\delta=(1,\delta_1,\delta_2,\ldots,\delta_{n-1},1)$.

\begin{definition}\label{defn:shifted_palindromic}
We call a sequence of integers $(a_0,a_1,\ldots,a_n)$ \emph{shifted palindromic} (or \emph{shifted symmetric}) if $a_0=1$ and $a_{i+1}=a_{n-i}$ for each $i\in\{0\ldots,n-1\}$.
\end{definition}

\begin{cor}[cf.~\protect{\cite[Corollary~1.5(ii)]{Kas03}}]\label{cor:gcd_are_terminal}
Let $X=\Proj(\lambda_0,\lambda_1,\ldots,\lambda_n)$ have at worst canonical singularities, such that $\gcd{\lambda_i,h}=1$ for all $i\in\{0,\ldots,n\}$. Then $X$ has at worst terminal singularities, and the corresponding simplex $P\subset\NQ$ has a shifted palindromic $\delta$-vector.
\end{cor}
\begin{proof}
Since
$$
\dmodpart{i}=0\quad\text{ if and only if }\quad
h\mid\kappa\lambda_i\quad\text{ if and only if }\quad
\frac{h}{\gcd{\lambda_i,h}}\ \Big|\ \kappa,$$
and since $\gcd{\lambda_i,h}=1$ by assumption, we conclude that $\modpart{i}\ne 0$ for any $\kappa\in\{1,\ldots,h-1\}$. In particular this tells us that $\fracpart{\lambda_i(h-\kappa)/h}=1-\modpart{i}$, and so
$$\sum_{i=1}^n\left\{\frac{\lambda_i(h-\kappa)}{h}\right\}=n+1-\sum_{i=0}^n\dmodpart{i}.$$
Hence $\delta_j=\delta_{n+1-j}$ and so $P$ has a shifted palindromic $\delta$-vector. Finally, since $X$ is canonical we have that $\delta_n=1$, hence $\delta_1=1$ and so $X$ is terminal.
\end{proof}

\begin{remark}
Suppose that $X=\Proj(\lambda_0,\lambda_1,\ldots,\lambda_n)$ is a terminal weighted projective space whose associated simplex $P=\sconv{\rho_0,\rho_1,\ldots,\rho_n}$ has shifted palindromic $\delta$-vector. Then the singularities of $X$ can be regarded as being ``as far from Gorenstein as possible''. By this we mean the following: Let $F_i:=\sconv{\rho_0,\ldots,\rho_{i-1},\widehat{\rho}_i,\rho_{i+1},\ldots,\rho_n}$ be the facet of $P$ not containing $\rho_i$. Then, in general, $\mathrm{height}(F_i)\cdot\Vol{F_i}=\lambda_i$. By~\cite[Theorem~2.1]{Hig10} $P$ has shifted palindromic $\delta$-vector if and only if $\Vol{F_i}=1$ for each $i$, hence $P$ has shifted palindromic $\delta$-vector if and only if each facet $F_i$ is at height $\lambda_i$. This stands in contrast to the Gorenstein case, where every facet is at height one.
\end{remark}

\section{Classifications in dimension four}\label{sec:dim4}
We are now almost in a position to classify all weighted projective spaces in dimension four with at worst terminal (or canonical) singularities. Before we can proceed, we require two additional results. First,~\cite[Theorem~3.5]{Kas08b} provides an upper bound on the ratio $\lambda_i/h$:

\begin{thm}[\protect{\cite[Theorem~3.5]{Kas08b}}]\label{thm:bounds}
Let $X=\Proj(\lambda_0,\lambda_1,\ldots,\lambda_n)$ be a weighted projective space with at worst canonical singularities, where the weights are ordered $\lambda_0\leq\lambda_1\leq\ldots\leq\lambda_n$.  Then, for any $i\in\{2,\ldots,n\}$, we have:
$$\frac{\lambda_i}{h}\leq\frac{1}{n-i+2}.$$
Furthermore, if $X$ has at worst terminal singularities then the inequalities are strict.
\end{thm}

For fixed $h$, Proposition~\ref{prop:terminal_coprimeness} and Theorem~\ref{thm:bounds} allow us to efficiently search through the candidate weights. Proposed weights can then be tested using Proposition~\ref{prop:terminal_sum} (or Proposition~\ref{prop:canonical_sum}). However, we still require a bound on $h$.

In the three-dimensional terminal case, a bound $h\leq 30$ was constructed in~\cite{Kas03} via a recursive search. Armed with the terminal classification,~\cite{Kas08a} extended this to a canonical classification via the so-called ``minimal'' simplices; this avoided the need to obtain a additional bound on $h$ in the canonical case. Rather than attempting to mirror this approach in dimension four, we shall make use of a recently announced result of Averkov, Kruempelmann, and Nill. By using the techniques developed in~\cite{Aver11}, they were able to generalise the bounds on reflexive simplices established in~\cite{Nill04}:

\begin{definition}\label{defn:sylvester_sequence}
The \emph{Sylvester sequence} is a sequence $y_i$ of positive integers defined by:
$$y_0:=2,\qquad y_m:=1+y_0\ldots y_{m-1}.$$
We also define the sequence $t_m:=y_m-1$, for all $m\in\Z_{\geq 0}$.
\end{definition}

\begin{thm}[Averkov, Kruempelmann, and Nill]\label{thm:akn}
Let $P\subset\NQ$ be an $n$-simplex containing only the origin in its interior,~i.e.~$\intr{P}\cap N=\{\orig\}$. If $n\geq 4$ then
$$\Vol{P}\leq 2t_{n-1}^2.$$
Furthermore this bound is sharp: $\Vol{P}=2t_{n-1}^2$ if and only if $P$ is isomorphic to the reflexive simplex dual to
$\Proj(1,1,2t_{n-1}/y_{n-2},\ldots,2t_{n-1}/y_0)$.
\end{thm}

When $P\subset\NQ$ is a polytope associated with weighted projective space, $\Vol{P}=h$. An application of Theorem~\ref{thm:akn} gives the bound $h<3528$ (the theorem allows us to exclude the case $h=3528$).

\begin{remark}
Notice that, for our purposes, Theorem~\ref{thm:akn} is not sharp. The simplices we consider are special in that their vertices generate the lattice $N$~\cite[Proposition~2]{BB92}.
\end{remark}

With the aid of a computer, a complete classification in dimension four is possible\footnote{The source code and classifications are available from \href{http://grdb.lboro.ac.uk/files/wps/dim4/}{\texttt{http://grdb.lboro.ac.uk/files/wps/dim4/}}.}:

\begin{thm}\label{thm:terminal_dim_4}
Let $(\lambda_0,\lambda_1,\lambda_2,\lambda_3,\lambda_4)\in\Z_{>0}^5$ be well-formed such that $X=\Proj(\lambda_0,\lambda_1,\lambda_2,\lambda_3,\lambda_4)$ has at worst terminal singularities. Then:
\begin{enumerate}
\item
Up to reordering, there are precisely $28,\!686$ possible weights;
\item\label{item:terminal_dim_4_max_vol}
The maximum value $h=881$ is attained only in the case $(20,21,123,287,430)$;
\item\label{item:terminal_dim_4_degree}
The maximum degree $(-K_X)^4=\frac{3418801}{1764}$ is attained only in the case $(1,1,6,14,21)$.
\end{enumerate}
\end{thm}

\begin{thm}\label{thm:canonical_dim_4}
Let $(\lambda_0,\lambda_1,\lambda_2,\lambda_3,\lambda_4)\in\Z_{>0}^5$ be well-formed such that $X=\Proj(\lambda_0,\lambda_1,\lambda_2,\lambda_3,\lambda_4)$ has at worst canonical singularities. Then:
\begin{enumerate}
\item
Up to reordering, there are precisely $338,\!752$ possible weights;
\item
The maximum value $h=3486$ is attained only in the case $(41,42,498,1162,1743)$;
\item\label{item:canonical_dim_4_degree}
The maximum degree $(-K_X)^4=3528$ is attained only in the case $(1,1,12,28,42)$.
\end{enumerate}
\end{thm}

Of course Theorem~\ref{thm:canonical_dim_4}\eqref{item:canonical_dim_4_degree} agrees with Theorem~\ref{thm:akn}. Whether the relationship between the weights of maximum degree in the terminal and canonical cases is coincidental, or whether it is indicative of a deeper result when $n\geq 4$, is unknown. We do note, however, that $\Proj(1,1,t_{n-1}/y_{n-2},\ldots,t_{n-1}/y_0)$ is always terminal:

\begin{lemma}\label{lem:maximum_terminal_candidate}
The weighted projective space $X=\Proj(1,1,t_{n-1}/y_{n-2},\ldots,t_{n-1}/y_0)$ of degree
$$(-K_X)^n=\frac{y_{n-1}^n}{t_{n-1}^{n-2}}$$
is terminal, for $n\geq 2$, and the associated simplex in $\NQ$ has shifted palindromic $\delta$-vector.
\end{lemma}
\begin{proof}
When $n<4$ we have $\Proj^2$ and $\Proj(1,1,2,3)$, both of which are terminal. Suppose that $n\geq 4$. We know from~\cite[Theorem~A]{Nill04} (or Theorem~\ref{thm:akn}) that $\Proj(1,1,2t_{n-1}/y_{n-2},\ldots,2t_{n-1}/y_0)$ is Gorenstein. Set $h'=1+1+2t_{n-1}/y_{n-2}+\ldots+2t_{n-1}/y_0$. Then $2t_{n-1}/y_i\mid h'$ for each $i\in\{0\ldots,n-2\}$. Now consider $h=1+1+t_{n-1}/y_{n-2}+\ldots+t_{n-1}/y_0$. Since $h'=2h-2$, we conclude that:
\begin{equation}\label{eq:divides_h-1}
\frac{t_{n-1}}{y_i}\ \Big|\ h-1.
\end{equation}
In particular, $\gcd{t_{n-1}/y_i,h}=1$. If we can show that $X$ is canonical, the result follows by Corollary~\ref{cor:gcd_are_terminal}.

Let $Y=\Proj(1,t_{n-1}/y_{n-2},\ldots,t_{n-1}/y_0)$, where the weights are given by the \emph{Sylvester weight system} of length $n-1$. This has sum of weights $h-1$, so by~\eqref{eq:divides_h-1} we see that $Y$ is Gorenstein. In particular, $Y$ is canonical. Let $P_X\subset\NQ$ and $P_Y\subset N'_{\Q}$ be simplices associated with $X$ and $Y$, respectively, and suppose that there exists a non-zero lattice point $v\in\intr{P_X}$. We will show that this implies the existence of a non-zero lattice point in $\intr{P_Y}$, contradicting $Y$ being canonical.

This is straight forward. After suitable change of basis we may assume that
$$\V{P_X}=\left\{e_1,\ldots,e_n,(-1,-t_{n-1}/y_{n-2},\ldots,-t_{n-1}/y_0)\right\},$$
where the $e_i$ are the standard basis elements for $N$. Similarly, we may assume that
$$\V{Q_X}=\left\{f_1,\ldots,f_{n-1},(-t_{n-1}/y_{n-2},\ldots,-t_{n-1}/y_0)\right\},$$
where the $f_i$ are the standard basis elements for $N'$. Since the vertices of $P_X$ generate the lattice, we can write
$$v=\mu_0e_1+\mu_1e_2+\ldots+\mu_{n-1}e_n+\mu_n(-1,-t_{n-1}/y_{n-2},\ldots,-t_{n-1}/y_0),$$
for some $\mu_i>0$, $\sum_{i=0}^n\mu_i=1$. Then
\begin{align*}
v'&=\mu_0\orig+\mu_1f_1+\ldots+\mu_{n-1}f_{n-1}+\mu_n(-t_{n-1}/y_{n-2},\ldots,-t_{n-1}/y_0)\\
&=\sum_{i=1}^{n-1}\left(\mu_i+\mu_0\frac{t_{n-1}}{y_{n-i-1}(h-1)}\right)f_i+\left(\mu_n+\mu_0\frac{1}{h-1}\right)(-t_{n-1}/y_{n-2},\ldots,-t_{n-1}/y_0)
\end{align*}
is a non-zero lattice point in $N'$ contained in the strict interior of $P_Y$.
\end{proof}

\begin{remark}\label{rem:comments_on_conjectures}
When $n=3$ the bound $(-K_X)^3\leq 72$, with equality given by $\Proj(1,1,1,3)$ and $\Proj(1,1,4,6)$, was conjectured by Fano and Iskovskikh in the case when $X$ is a canonical Gorenstein Fano variety, and proven by Prokhorov~\cite{Pro05}. It was also shown to hold in the case when $X$ is a (not necessarily Gorenstein) canonical toric Fano variety in~\cite[Theorem~3.6]{Kas08a}.

If $n\geq 4$ and $X$ is a canonical Gorenstein Fano variety, the conjecture that $(-K_X)^n\leq 2t_{n-1}^2$, with equality only in the case $\Proj(1,1,2t_{n-1}/y_{n-2},\ldots,2t_{n-1}/y_0)$, was proposed by Nill~\cite[Conjecture~2.1]{Nill04}. The classification of the $4$-dimensional reflexive polytopes~\cite{KS00} confirms this holds for all Gorenstein toric Fano $4$-folds.

In the terminal case, we notice that the largest degree for a terminal Gorenstein toric Fano $4$-fold is $800$: this occurs only once, for the reflexive polytope
$$\sconv{(1,0,0,0),(0,1,0,0),(0,0,1,0),(0,0,0,\pm1),(-1,-1,-1,-3)}.$$
(In fact the corresponding toric variety is non-singular.) This does not contradict the obvious conjecture that the weighted projective space in Lemma~\ref{lem:maximum_terminal_candidate} is the unique terminal toric Fano $n$-fold of maximum degree.
\end{remark}

\subsection*{One-point lattice simplices}
Given the classification of terminal weights, it is possible to classify all one-point lattice simplices, up to lattice translation and change of basis.

\begin{definition}
Let $P\subset\NQ$ be an $n$-simplex such that $\partial P\cap N=\V{P}$ (such polytopes are often referred to as~\emph{clean}) and $\abs{\intr{P}\cap N}=1$. Then we call $P$ a \emph{one-point lattice simplex}.
\end{definition}

Because of their natural combinatorial definition, one-point lattice simplices -- and their generalisation to $k$-point lattice simplices -- have been studied extensively (for example,~\cite{Rez86,Rez,Duo08}). They also have a natural interpretation in terms of toric geometry: After translation we may insist that $\intr{P}\cap N=\{\orig\}$, hence $P$ corresponds to a terminal \emph{fake weighted projective space}~\cite{Buc08,Kas08b}. That is, the spanning fan of $P$ corresponds to a rank-one $\Q$-factorial toric variety $X$ with at worst terminal singularities. The three-dimension one-point lattice simplices were first classified in~\cite{Kas03} in this context.

Let $\V{P}=\{\rho_0,\rho_1,\ldots,\rho_n\}\subset N$. Since $\orig\in\intr{P}$ there exists a (unique) positive collection of weights $(\lambda_0,\lambda_1,\ldots,\lambda_n)\in\Z_{>0}^{n+1}$, $\gcd{\lambda_0,\lambda_1,\ldots,\lambda_n}=1$, such that $\sum_{i=0}^n\lambda_i\rho_i=\orig$. Define the \emph{multiplicity} of $P$ to be the index $\mult{P}:=[N:\Z\rho_0+\Z\rho_1+\ldots+\Z\rho_n]$ of the sublattice generated by the vertices.

\begin{thm}[\cite{BB92,Con02,Buc08,Kas08b}]\label{thm:fake_wps}
With notation as above, let $Y=\Proj(\lambda_0,\lambda_1,\ldots,\lambda_n)$ with associated simplex $Q\subset\NQ$. Then:
\begin{enumerate}
\item
$X\cong Y$ if and only if $\mult{P}=1$;
\item
$X$ is the quotient of $Y$ by the action of a finite group of order $\mult{P}$ acting free in codimension one;
\item
There exists a Hermite normal form $H$ with determinant $\det{H}=\mult{P}$ such that $P\cong QH$;
\item
If $P$ is terminal (resp.~canonical) then $Q$ is terminal (resp.~canonical) and
$$\mult{P}\leq\frac{h^{n-1}}{\lambda_1\lambda_2\ldots\lambda_n}=\frac{\lambda_0}{h}(-K_Y)^n.$$
\end{enumerate}
\end{thm}

\begin{remark}
A consequence of Theorem~\ref{thm:fake_wps} is that $\Vol{P}=\mult{P}\Vol{Q}$. Dualising we obtain $(-K_X)^n=(-K_Y)^n/\,\mult{P}$. Hence the maximum degree amongst all fake weighted projective spaces of fixed dimension and prescribed singularities is always attained by a weighted projective space $Y$.
\end{remark}

Theorems~\ref{thm:terminal_dim_4} and~\ref{thm:fake_wps} allow us to make a complete classification of all four-dimensional one-point lattice simplices (or, equivalently, all four-dimensional terminal fake weighted projective spaces). We summarise this classification\footnote{The complete classification is available from \href{http://grdb.lboro.ac.uk/files/wps/dim4/}{\texttt{http://grdb.lboro.ac.uk/files/wps/dim4/}}.}:

\begin{thm}\label{thm:one_point_classification}
There exist $35,\!947$ equivalence classes of one-point lattice $4$-simplices.
\begin{enumerate}
\item
Exactly $7,\!261$ of these equivalence classes have $\mult{P}>1$.
\item
The largest occurring multiplicity is $41$; this happens only in the case
$$\sconv{(1,0,0,0),(0,1,0,0),(0,0,1,0),(3,10,28,41),(-4,-11,-29,-41)}$$
with weights $(1,1,1,1,1)$.
\item\label{item:one_point_classification_max_vol}
The largest occurring volume is $881$; this volume is obtained exactly once, by the simplex corresponding to $\Proj(20,21,123,287,430)$.
\end{enumerate}
\end{thm}

That the bound on the volume given by Theorem~\ref{thm:terminal_dim_4}\eqref{item:terminal_dim_4_max_vol} and by Theorem~\ref{thm:one_point_classification}\eqref{item:one_point_classification_max_vol} agree is unexpected. We observe that the simplex has shifted palindromic $\delta$-vector $(1,1,439,439,1)$.

\section{Terminal Gorenstein weighted projective space}\label{sec:terminal_reflexive_general}
Let $(\lambda_0,\lambda_1,\ldots,\lambda_n)\in\Z_{>0}^{n+1}$ be well-formed. In what follows we shall require that $\lambda_i\mid h$ for each $i\in\{0,\ldots,n\}$. Equivalently, the associated weighted projective space $\Proj(\lambda_0,\lambda_1,\ldots,\lambda_n)$ is Gorenstein, that is, that the corresponding $n$-simplex $P\subset\NQ$ is reflexive.

\begin{definition}
Suppose that $\lambda_i\mid h$ for each $i\in\{0,\ldots,n\}$. We define
$$s(\kappa):=\abs{\big\{i\in\{0,\ldots,n\}\mid\modpart{i}=0\big\}}.$$
\end{definition}

\begin{lemma}\label{lem:reflexive_sum}
Suppose that $\lambda_i\mid h$ for each $i\in\{0,\ldots,n\}$. Then:
\begin{enumerate}
\item\label{item:reflexive_sum_1}
$\sum_{i=0}^n\fracpart{\lambda_i(\kappa+1)/h}=\sum_{i=1}^n\modpart{i}+1-s(\kappa+1)$;
\item\label{item:reflexive_sum_2}
$\sum_{i=0}^n\invmodpart{i}=n+1-s(\kappa)-\sum_{i=0}^n\modpart{i}$.
\end{enumerate}
\end{lemma}
\begin{proof}
Set $k=s(\kappa+1)$. We may assume that, after possible reordering of the $\lambda_i$, $\fracpart{\lambda_i(\kappa+1)/h}=0$ for $i\in\{0,\ldots,k-1\}$. Then $\modpart{i}=1-\lambda_i/h$ for each $i\in\{0,\ldots,k-1\}$, and $\fracpart{\lambda_i(\kappa+1)/h}=\modpart{i}+\lambda_i/h$ for each $i\in\{k,\ldots,n\}$. Hence:
\begin{align*}
\sum_{i=0}^n\fracpart{\frac{\lambda_i(\kappa+1)}{h}}&=\sum_{i=0}^{k-1}\fracpart{\frac{\lambda_i(\kappa+1)}{h}}+\sum_{i=k}^n\fracpart{\frac{\lambda_i(\kappa+1)}{h}}\\
&=\sum_{i=k}^n\left(\dmodpart{i}+\frac{\lambda_i}{h}\right)\\
&=\sum_{i=0}^n\dmodpart{i}+1-k.
\end{align*}
To prove~\eqref{item:reflexive_sum_2}, notice that:
$$\dinvmodpart{i}=\left\{\begin{array}{ll}
1-\modpart{i},&\text{ if }\modpart{i}\ne 0,\\
0,&\text{ otherwise.}
\end{array}\right.$$
Hence:
$$\sum_{i=0}^n\dinvmodpart{i}=n+1-s(\kappa)-\sum_{i=0}^n\dmodpart{i}.$$
\end{proof}

\begin{definition}
Suppose that $\lambda_i\mid h$. We say that $\kappa$ is a \emph{jump point} for $\lambda_i$ if $\modpart{i}=0$, and that the jump points of $\lambda_i$ are of \emph{order} $k_i$, where $k_i:=h/\lambda_i\in\Z$.
\end{definition}

\begin{cor}\label{cor:reflex_zeros}
Suppose that $\lambda_i\mid h$ for each $i\in\{0,\ldots,n\}$ and that the corresponding Gorenstein weighted projective space has at worst terminal singularities. Then $s(\kappa)\leq n-3$ for all $\kappa\in\{2,\ldots,h-2\}$. In particular,
$$h=\lcm{k_0,\ldots,k_{i-1},\widehat{k}_{i},k_{i+1},\ldots,k_{j-1},\widehat{k}_{j},k_{j+1},\ldots,k_{l-1},\widehat{k}_{l},k_{l+1},\ldots,k_n}$$
for all sequences of length $n-2$ given by omitting $k_i$, $k_j$, and $k_l$, where $0\leq i<j<l\leq n$.
\end{cor}
\begin{proof}
By Proposition~\ref{prop:terminal_sum} we know that, for each $\kappa\in\{2,\ldots,h-2\}$, $\sum_{i=0}^n\modpart{i}\geq 2$, hence $s(\kappa)\leq n-3$. As a consequence,
$$h-2<\lcm{k_0,\ldots,k_{i-1},\widehat{k}_{i},k_{i+1},\ldots,k_{j-1},\widehat{k}_{j},k_{j+1},\ldots,k_{l-1},\widehat{k}_{l},k_{l+1},\ldots,k_n}.$$
But $\lcm{k_0,\ldots,k_n}=h$, hence
$$h=\lcm{k_0,\ldots,k_{i-1},\widehat{k}_{i},k_{i+1},\ldots,k_{j-1},\widehat{k}_{j},k_{j+1},\ldots,k_{l-1},\widehat{k}_{l},k_{l+1},\ldots,k_n}.$$
\end{proof}

\begin{prop}\label{prop:reflex_terminal_sum}
Let $(\lambda_0,\lambda_1,\ldots,\lambda_n)\in\Z_{>0}^{n+1}$ be well-formed such that $\lambda_i\mid h$ for each $i\in\{0,\ldots,n\}$. The weighted projective space $\Proj(\lambda_0,\lambda_1,\ldots,\lambda_n)$ has at worst terminal Gorenstein singularities if and only if $\sum_{i=0}^n\modpart{i}\in\{2,\ldots,n-2\}$ for each $\kappa\in\{2,\ldots,h-3\}$.
\end{prop}
\begin{proof}
Since the $\lambda_i\mid h$, we know that $X=\Proj(\lambda_0,\lambda_1,\ldots,\lambda_n)$ is Gorenstein. Suppose that $X$ has at worst terminal singularities. By Proposition~\ref{prop:terminal_sum} we know that $2\leq\sum_{i=0}^n\modpart{i}\leq n-1$ for all $\kappa\in\{2,\ldots,h-2\}$. Suppose that $\sum_{i=0}^n\modpart{i}=n-1$ for some $\kappa\in\{2,\ldots,h-3\}$. By Lemma~\ref{lem:reflexive_sum}\eqref{item:reflexive_sum_1} we have that $s(\kappa+1)+\sum_{i=0}^n\fracpart{\lambda_i(\kappa+1)/h}=n$, and by Lemma~\ref{lem:reflexive_sum}\eqref{item:reflexive_sum_2} we see that $\sum_{i=0}^n\fracpart{\lambda_i(h-\kappa-1)/h}=1$, which is a contradiction.

Conversely suppose that $\sum_{i=0}^n\modpart{i}\in\{2,\ldots,n-2\}$ for each $\kappa\in\{2,\ldots,h-3\}$. In order to apply Proposition~\ref{prop:terminal_sum}, and so conclude that $X$ is terminal, we need to show that $\sum_{i=0}^n\fracpart{\lambda_i(h-2)/h}\in\{2,\ldots,n-1\}$. Since $\sum_{i=0}^n\fracpart{2\lambda_i/h}=2$, we conclude that $s(2)=0$ by Lemma~\ref{lem:reflexive_sum}\eqref{item:reflexive_sum_1}. Hence $\sum_{i=0}^n\fracpart{\lambda_i(h-2)/h}=n-1$ by Lemma~\ref{lem:reflexive_sum}\eqref{item:reflexive_sum_2}.
\end{proof}

\begin{thm}\label{thm:gorenstein_terminal_dim_4}
If $X=\Proj(\lambda_0,\lambda_1,\lambda_2,\lambda_3,\lambda_4)$ is a four-dimensional terminal Gorenstein weighted projective space then $X$ is isomorphic to either $\Proj^4$ or $\Proj(1,1,1,1,2)$.
\end{thm}
\begin{proof}
Without loss of generality we shall assume that $\lambda_0\leq\ldots\le\lambda_4$. Proposition~\ref{prop:reflex_terminal_sum} gives that $\sum_{i=0}^4\modpart{i}=2$ for $\kappa\in\{2,\ldots,h-3\}$. Applying this to Lemma~\ref{lem:reflexive_sum}\eqref{item:reflexive_sum_2} yields $s(\kappa)=1$ for $\kappa\in\{3,\ldots,h-3\}$, and that $s(2)=0$.

Suppose that $h\geq 6$. Since $s(3)=1$, it must be that the jump points for $\lambda_4$ have order $3$. In particular, $3\mid h$. Let us assume for a contradiction that $h\geq 9$. Since $s(2)=0$ and $s(4)=1$, and since the jump points for $\lambda_4$ have order $3\notdivides 4$, we have that the jump points for $\lambda_3$ must have order $4$. Similarly, since $s(5)=1$, we see that the jump points for $\lambda_2$ have order $5$. Since $3\mid h$ and $5\mid h$ we have that $h\geq 15$. But $\lcm{3,4}=12$, which tells us that $s(12)\geq 2$, contradicting the requirement that $s(12)=1$.

Hence we have that $h\leq 6$. There are only two possibilities: either $h=5$ or $h=6$. In the former case we obtain $\lambda_0=\ldots=\lambda_4=1$, corresponding to $\Proj^4$. The latter case gives weights $\lambda_0=\ldots=\lambda_3=1$ and $\lambda_4=2$, corresponding to $\Proj(1,1,1,1,2)$. 
\end{proof}

\begin{thm}\label{thm:gorenstein_terminal_dim_5}
If $X=\Proj(\lambda_0,\lambda_1,\ldots,\lambda_5)$ is a five-dimensional terminal Gorenstein weighted projective space then $X$ is isomorphic to one of $\Proj^5$, $\Proj(1^4,2^2)$, $\Proj(1^2,2^2,3^2)$, or $\Proj(1^3,2,3,4)$.
\end{thm}
\begin{proof}
We may assume that the weights are in increasing order, $\lambda_0\leq\ldots\leq\lambda_5$. With notation as above, let $k_i\in\Z$ be such that $k_i\lambda_i=h$ for each $i\in\{0,\ldots,5\}$. Theorem~\ref{thm:bounds} tells us that $k_i>7-i$, for $i\in\{2,3,4,5\}$. In particular $k_5>2, k_4>3$ and $k_3>4$. Proposition~\ref{prop:reflex_terminal_sum} gives that $\sum_{i=0}^5\modpart{i}\in\{2,3\}$ for $\kappa\in\{2,\ldots,h-3\}$, and Corollary~\ref{cor:reflex_zeros} tells us that $h=\lcm{k_5,k_4,k_3}$.

We shall assume that $\lambda_5\neq 1$. Thus we have that $h\geq 7$. If $k_5>4$ then, by Lemma~\ref{lem:reflexive_sum}\eqref{item:reflexive_sum_1}, $\sum_{i=0}^5\fracpart{4\lambda_i/h}=4>3$, a contradiction. Hence $k_5=3$ or $4$.

\begin{table}[tb]
\centering
\caption{The ten candidate values of $k_5$, $k_4$, and $k_3$ found in the proof of Theorem~\ref{thm:gorenstein_terminal_dim_5}.}
\label{tab:k_values_term_dim_5}
\begin{tabular}{cc}
\toprule
\begin{tabular}{ccc>{\centering\arraybackslash}m{1.75em}}
$k_5$&$k_4$&$k_3$&$h$\\
\cmidrule(lr){1-3}\cmidrule(lr){4-4}
\oddrow $3$&$4$&$5$&$60$\\
\evnrow $3$&$4$&$6$&$12$\\
\oddrow $3$&$4$&$7$&$84$\\
\evnrow $3$&$5$&$5$&$15$\\
\oddrow $3$&$5$&$6$&$30$\\
\end{tabular}&
\begin{tabular}{ccc>{\centering\arraybackslash}m{1.75em}}
$k_5$&$k_4$&$k_3$&$h$\\
\cmidrule(lr){1-3}\cmidrule(lr){4-4}
\oddrow $3$&$5$&$7$&$105$\\
\evnrow $4$&$4$&$5$&$20$\\
\oddrow $4$&$4$&$6$&$12$\\
\evnrow $4$&$5$&$5$&$20$\\
\oddrow $4$&$5$&$6$&$60$\\
\end{tabular}\\
\bottomrule
\end{tabular}
\end{table}

\begin{description}
\item[$k_5=3$]
We have that $3\mid h$, and so $h\geq 9$. Now $\sum_{i=0}^5\modpart{i}\leq 3$ for $\kappa\in\{2,\ldots,6\}.$ If $k_4\geq 6$ then $\sum_{i=0}^5\fracpart{5\lambda_i/h}=4$. Hence $k_4=4$ or $5$.
\begin{description}
\item[$k_4=4$]
Since $3\mid h$ and $4\mid h$ we have that $h\geq 12$. If $k_3\geq 8$ then $\sum_{i=0}^5\fracpart{7\lambda_i/h}=4$. Hence $k_3=5,6,$ or $7$.
\item[$k_4=5$]
In this case we see that $h\geq 15$. If $k_3\geq 8$ then $\sum_{i=0}^5\fracpart{7\lambda_i/h}=4$. Hence $k_3=5,6,$ or $7$.
\end{description}
\item[$k_5=4$]
We have the $4\mid h$ and so $h\geq 8$. Now if $k_4\geq 6$ then $\sum_{i=0}^5\fracpart{5\lambda_i/h}=4$. Hence $k_4=4$ or $5$.
\begin{description}
\item[$k_4=4$]
Suppose that $\lambda_3\neq 1$. Then $h\geq 12$. If $h_3\geq 7$ then $\sum_{i=0}^5\fracpart{6\lambda_i/h}=4$. Hence $k_3=5$ or $6$.
\item[$k_4=5$]
We have that $h\geq 20$. In $k_3\geq 7$ then $\sum_{i=0}^5\fracpart{6\lambda_i/h}=4$. Hence $k_3=5$ or $6$.
\end{description}
\end{description}

These results are summarised in Table~\ref{tab:k_values_term_dim_5}.

Knowing $h$, the values of $\lambda_5,\lambda_4$ and $\lambda_3$ are trivial to calculate. Hence we also know the value of the sum $\lambda_0+\lambda_1+\lambda_2$. Since $\lambda_2<h/5$ by Theorem~\ref{thm:bounds}, $\lambda_2\leq\lambda_3$, and $\lambda_i\mid h$, it is an easy task to calculate all possible values $\lambda_0\leq\lambda_1\leq\lambda_2$.

For the sake of completeness we shall reproduce this calculation here in two cases, leaving the remainder to the reader.

\begin{table}[tb]
\centering
\caption{The seventeen candidate weights for the Gorenstein weighted projective spaces in dimension $5$ with at worst terminal singularities, as found in the proof of Theorem~\ref{thm:gorenstein_terminal_dim_5}.}
\label{tab:gorenstein_terminal_dim_5}
\begin{tabular}{cc}
\toprule
\begin{tabular}{ccccccc}
$\lambda_0$&$\lambda_1$&$\lambda_2$&$\lambda_3$&$\lambda_4$&$\lambda_5$&$h$\\
\cmidrule(lr){1-6}\cmidrule(lr){7-7}
\oddrow $1$&$1$&$1$&$1$&$1$&$1$&$6$\\
\evnrow $1$&$1$&$1$&$1$&$1$&$1$&$6$\\
\oddrow $1$&$1$&$1$&$1$&$2$&$2$&$8$\\
\evnrow $1$&$1$&$1$&$2$&$3$&$4$&$12$\\
\oddrow $1$&$1$&$2$&$2$&$3$&$3$&$12$\\
\evnrow $2$&$2$&$2$&$4$&$5$&$5$&$20$\\
\oddrow $1$&$3$&$5$&$5$&$6$&$10$&$30$\\
\evnrow $2$&$2$&$5$&$5$&$6$&$10$&$30$\\
\oddrow $3$&$3$&$3$&$5$&$6$&$10$&$30$\\
\end{tabular}&
\begin{tabular}{ccccccc}
$\lambda_0$&$\lambda_1$&$\lambda_2$&$\lambda_3$&$\lambda_4$&$\lambda_5$&$h$\\
\cmidrule(lr){1-6}\cmidrule(lr){7-7}
\oddrow $1$&$2$&$10$&$12$&$15$&$20$&$60$\\
\evnrow $1$&$6$&$6$&$12$&$15$&$20$&$60$\\
\oddrow $2$&$5$&$6$&$12$&$15$&$20$&$60$\\
\evnrow $3$&$4$&$6$&$12$&$15$&$20$&$60$\\
\oddrow $3$&$5$&$5$&$12$&$15$&$20$&$60$\\
\evnrow $3$&$10$&$10$&$10$&$12$&$15$&$60$\\
\oddrow $4$&$4$&$5$&$12$&$15$&$20$&$60$\\
\evnrow $4$&$7$&$12$&$12$&$21$&$28$&$84$\\
\\
\end{tabular}\\
\bottomrule
\end{tabular}
\end{table}

\begin{description}
\item[$h=84$]
We have that $k_5=3$, $k_4=4$, and $k_3=7$. Hence $\lambda_5=28$, $\lambda_4=21$, and $\lambda_3=12$. This means that $\lambda_0+\lambda_1+\lambda_2=23$, and $\lambda_0\leq\lambda_1\leq\lambda_2\leq\lambda_3=12$. The only possibility is $\lambda_0=4$, $\lambda_1=7$, and $\lambda_2=12$.

\item[$h=105$]
We have that $k_5=3$, $k_4=5$, and $k_3=7$, giving $\lambda_5=35$, $\lambda_4=21$, and $\lambda_3=15$. We have that $\lambda_0+\lambda_1+\lambda_2=34$, with $\lambda_0\leq\lambda_1\leq\lambda_2\leq\lambda_3=15$. Since the only positive integers dividing $105$ which are at most $15$ are $15$, $7$, $5$, $3$, and $1$, we see that a total $34$ cannot be made. Hence we must rule out this possibility.
\end{description}

The case when $h=15$ is similarly impossible. The resulting candidate weights are collected in Table~\ref{tab:gorenstein_terminal_dim_5}. All but five of the weights fail to satisfy Proposition~\ref{prop:terminal_coprimeness}; in order to rule out the case $(1,3,5,5,6,10)$ we notice that $\lcm{30/5,30/5,30/10}=6\neq 30$, contradicting Corollary~\ref{cor:reflex_zeros}. Finally, the remaining four weights can be verified to be terminal by Proposition~\ref{prop:reflex_terminal_sum}.
\end{proof}

\section{Computer-generated classifications in higher dimensions}\label{sec:terminal_reflexive_computer}
Our aim is to construct well-formed weights $\lambda_0\leq\lambda_1\leq\ldots\leq\lambda_n$, $\lambda_i\mid h$, such that the corresponding Gorenstein weighted projective space has at worst terminal singularities. We begin by establishing bounds on the $k_i$.

\begin{cor}\label{cor:order_lambda_n}
Let $X=\Proj(\lambda_0,\lambda_1,\ldots,\lambda_n)$ be a terminal Gorenstein weighted projective space with $\lambda_n\geq\lambda_i$, $\lambda_n\neq 1$. Then $k_n\leq n-1$.
\end{cor}
\begin{proof}
We know that $\sum_{i=0}^n\fracpart{\lambda_i/h}=1$. Since $\lambda_n\neq 1$ we know that $h\geq n+2$. Hence $h-3\geq n-1$. Proposition~\ref{prop:reflex_terminal_sum} tells us that $\sum_{i=0}^n\modpart{i}\leq n-2$ for $\kappa\in\{2,\ldots,n-1\}$. Hence, by Lemma~\ref{lem:reflexive_sum}\eqref{item:reflexive_sum_1}, there must exist a $\kappa\in\{2,\ldots,n-1\}$ such that $s(\kappa)>0$. In particular, $k_n\leq n-1$.
\end{proof}

The following proposition can be used to provide an inductive bound on the $k_i$. Notice that the statement makes no distinction between terminal and canonical cases.

\begin{prop}[\protect{\cite[Theorem~9.8.4]{KasPhD}}]\label{prop:order_lambda_general}
Suppose that $\Proj(\lambda_0,\lambda_1,\ldots,\lambda_n)$ is a Gorenstein weighted projective space, with weights $\lambda_0\leq\lambda_1\leq\ldots\leq\lambda_n$. Then
$$k_i\leq\frac{i+1}{1-\sum\limits_{\scriptscriptstyle{j=i+1}}^{\scriptscriptstyle{n}}\frac{1}{k_j}},\qquad\text{for }i\in\{0,\ldots,n\}.$$
\end{prop}
\begin{proof}
We shall prove by induction that:
\begin{equation}\label{eq:order_lambda_general}
(i+1)\prod_{j=i+1}^nk_j\geq\Bigg(\prod_{j=i+1}^nk_j-\sum_{j=i+1}^n\mathop{\prod_{l=i+1}^n}_{l\neq j}k_l\Bigg)k_i,\qquad\text{for all }i\in\{0,\ldots,n\}.
\end{equation}
Recall that $h=\sum_{i=0}^n\lambda_i$, so that:
\begin{equation}\label{eq:egyptian_fracs}
\sum_{i=0}^n\frac{1}{k_i}=1.
\end{equation}
Rearranging gives us:
$$\prod_{i=0}^nk_i=\sum_{i=0}^n\mathop{\prod_{j=0}^n}_{j\neq i}k_j.$$
Collecting together the $k_0$ terms gives what will form the base case of our induction:
$$\prod_{i=1}^nk_i=\Bigg(\prod_{i=1}^nk_i-\sum_{i=1}^n\mathop{\prod_{j=1}^n}_{j\neq i}k_j\Bigg)k_0.$$
Now suppose that, for some $m\in\{0,\ldots,n-1\}$, we have:
$$(m+1)\prod_{i=m+1}^nk_i\geq\Bigg(\prod_{i=m+1}^nk_i-\sum_{i=m+1}^n\mathop{\prod_{j=m+1}^n}_{j\neq i}k_j\Bigg)k_m.$$
Since $k_m\geq k_{m+1}>0$ we obtain:
$$(m+1)\prod_{i=m+2}^nk_i\geq\prod_{i=m+1}^nk_i-\sum_{i=m+1}^n\mathop{\prod_{j=m+1}^n}_{j\neq i}k_j.$$
Collecting together the $k_{m+1}$ terms on the right of the inequality gives:
$$(m+2)\prod_{i=m+2}^nk_i\geq\Bigg(\prod_{i=m+2}^nk_i-\sum_{i=m+2}^n\mathop{\prod_{j=m+2}^n}_{j\neq i}k_j\Bigg)k_{m+1}.$$
Thus we have proved~\eqref{eq:order_lambda_general}.

Finally, observe that $\prod_{j=i+1}^nk_j>0$, and so~\eqref{eq:order_lambda_general} gives us:
$$i+1\geq\bigg(1-\sum_{j=i+1}^n\frac{1}{k_j}\bigg)k_i.$$
Now $k_0>0$, and so equation~\eqref{eq:egyptian_fracs} tells us that the term in parenthesis is positive. Dividing through yields the result.
\end{proof}

We will assume that $n\geq 3$ and that $\lambda_n>1$ (so that the weights corresponding to $\Proj^n$ are excluded from what follows). It makes sense to split the computation of possible weights into two parts: first we calculate the possible values for $k_n\leq k_{n-1}\leq\ldots\leq k_3$; second, we calculate the values for $\lambda_2\geq\lambda_1\geq\lambda_0$.

The possible $k_i$, $i\geq 3$, are calculated inductively. Theorem~\ref{thm:bounds} provides the general lower bound $k_i\geq n-i+3$, Corollary~\ref{cor:order_lambda_n} provides an upper bound for $k_n$, and $k_i$, $i\neq n$ is bounded above by Proposition~\ref{prop:order_lambda_general}:
\begin{align*}
3&\leq k_n\leq n-1,\\
\max{k_{i+1},n-i+3}&\leq k_i\leq\frac{i+1}{1-\sum\limits_{\scriptscriptstyle{j=i+1}}^{\scriptscriptstyle{n}}\frac{1}{k_j}},\qquad\text{for }i\in\{3,\ldots,n-1\}.
\end{align*}

By exploiting the fact that $k_n\leq k_{n-1}\leq\ldots\leq k_0$ is an increasing sequence, the upper bound $\mathrm{max}_i$ for $k_i$ can potentially be refined. By Lemma~\ref{lem:reflexive_sum}\eqref{item:reflexive_sum_1} we know the value of $\sum_{i=0}^n\modpart{i}$ for $\kappa\in\{2,\ldots,k_{i+1}-1\}$. We can attempt to extend this range out to $\mathrm{max}_i-1$ and look for the first contradiction to Proposition~\ref{prop:terminal_sum} or Proposition~\ref{prop:reflex_terminal_sum}. We describe this process in psedo-code:

\begin{center}
\vspace{0.6em}
\begin{minipage}{0.9\columnwidth}
\hrulefill
\smaller\begin{algorithmic}
\State $S\gets 1$
\For{$\kappa=2$ \ForTo $\mathrm{max}_i-1$}
\State $S\gets S+1$
\For{$j=i+1$ \ForTo $n$}
\If{$k_j\mid\kappa$}
\State $S\gets S-1$
\EndIf
\EndFor
\If{$S>n-1$ \AndAlso $\kappa<\mathrm{max}_i-2$}
\State{$\mathrm{max}_i\gets\kappa+2$}
\ElsIf{$S>n-2$ \AndAlso $\kappa<\mathrm{max}_i-3$}
\State{$\mathrm{max}_i\gets\kappa+3$}
\ElsIf{$S<2$ \AndAlso $\kappa<\mathrm{max}_i-2$}
\State{$\mathrm{max}_i\gets\kappa+2$}
\EndIf
\EndFor
\end{algorithmic}
\hrulefill
\end{minipage}
\vspace{0.6em}
\end{center}

Once a choice of $k_n\leq k_{n-1}\leq\ldots\leq k_3$ has been made, $h$ can be recovered by Corollary~\ref{cor:reflex_zeros}. Hence $\lambda_n\geq\lambda_{n-1}\geq\ldots\geq\lambda_3$ and the sum $\lambda_0+\lambda_1+\lambda_2$ are known. Calculating the possible choices for $\lambda_2\geq\lambda_1\geq\lambda_0$ is then a finite search. Of course the resulting candidate weights $(\lambda_0,\lambda_1,\ldots,\lambda_n)$ might not correspond to a terminal weighted projective space: this can be checked using Proposition~\ref{prop:reflex_terminal_sum}.

\begin{table}[tb]
\centering
\caption{The eighteen possible weights for terminal Gorenstein weighted projective space in dimension six.}
\label{tab:gorenstein_terminal_dim_6}
\begin{tabular}{cc}
\toprule
\begin{tabular}{cccccccc}
$\lambda_0$&$\lambda_1$&$\lambda_2$&$\lambda_3$&$\lambda_4$&$\lambda_5$&$\lambda_6$&$h$\\
\cmidrule(lr){1-7}\cmidrule(lr){8-8}
\oddrow $1$&$1$&$1$&$1$&$1$&$1$&$1$&$7$\\
\evnrow $1$&$1$&$1$&$1$&$1$&$1$&$2$&$8$\\
\oddrow $1$&$1$&$1$&$1$&$1$&$1$&$3$&$9$\\
\evnrow $1$&$1$&$1$&$1$&$2$&$2$&$2$&$10$\\
\oddrow $1$&$1$&$1$&$1$&$1$&$3$&$4$&$12$\\
\evnrow $1$&$1$&$1$&$1$&$2$&$2$&$4$&$12$\\
\oddrow $1$&$1$&$1$&$1$&$2$&$3$&$3$&$12$\\
\evnrow $1$&$1$&$1$&$2$&$2$&$2$&$3$&$12$\\
\oddrow $1$&$1$&$1$&$1$&$3$&$3$&$5$&$15$\\
\end{tabular}&
\begin{tabular}{cccccccc}
$\lambda_0$&$\lambda_1$&$\lambda_2$&$\lambda_3$&$\lambda_4$&$\lambda_5$&$\lambda_6$&$h$\\
\cmidrule(lr){1-7}\cmidrule(lr){8-8}
\oddrow $1$&$1$&$2$&$2$&$3$&$3$&$6$&$18$\\
\evnrow $1$&$1$&$2$&$2$&$4$&$5$&$5$&$20$\\
\oddrow $1$&$1$&$1$&$3$&$4$&$6$&$8$&$24$\\
\evnrow $1$&$1$&$2$&$3$&$3$&$6$&$8$&$24$\\
\oddrow $1$&$1$&$3$&$3$&$4$&$4$&$8$&$24$\\
\evnrow $1$&$2$&$3$&$3$&$3$&$4$&$8$&$24$\\
\oddrow $1$&$2$&$3$&$3$&$5$&$6$&$10$&$30$\\
\evnrow $2$&$2$&$3$&$3$&$5$&$5$&$10$&$30$\\
\oddrow $1$&$3$&$4$&$5$&$12$&$15$&$20$&$60$\\
\end{tabular}\\
\bottomrule
\end{tabular}
\end{table}

\begin{thm}\label{thm:gorenstein_terminal_dim_6}
Let $X$ be a six-dimensional terminal Gorenstein weighted projective space. Up to reordering, $X$ has one of the eighteen weights given in Table~\ref{tab:gorenstein_terminal_dim_6}.
\end{thm}

The number of terminal Gorenstein weighted projective spaces grows surprisingly slowly\footnote{The source code and classifications are available from \href{http://grdb.lboro.ac.uk/files/wps/reflexive/}{\texttt{http://grdb.lboro.ac.uk/files/wps/reflexive/}}.}. Table~\ref{tab:terminal_list} records this growth; the largest degree $(-K_X)^n$ and corresponding weights are recorded in Table~\ref{tab:max_deg}; the weights attaining the largest occurring value of $h$ are recorded in Table~\ref{tab:max_h}.

\begin{table}[tb]
\centering
\caption{The weights of the terminal Gorenstein weighted projective space in dimension $n$ that achieve the maximum degree $(-K_X)^n$.}
\label{tab:max_deg}
\begin{tabular}{ccl}
\toprule
$n$&$(-K_X)^n$&\multicolumn{1}{c}{Weights}\\
\cmidrule(lr){1-1}\cmidrule(lr){2-2}\cmidrule(lr){3-3}
\oddrow $4$&$648$&$(1,1,1,1,2)$\\
\evnrow $5$&$10368$&$(1,1,1,2,3,4)$\\
\oddrow $6$&$331776$&$(1,1,1,3,4,6,8)$\\
\evnrow $7$&$49787136$&$(1,1,1,6,12,14,21,28)$\\
\oddrow $8$&$21781872000$&$(1,1,4,10,15,60,84,105,140)$\\
\evnrow $9$&$340424620687872$&$(1,1,1,42,84,258,516,602,903,1204)$\\
\oddrow $10$&$23029100604532998144$&$(1,1,1,156,312,1884,3768,6123,8164,12246,16328)$\\
\bottomrule
\end{tabular}
\end{table}

\begin{table}[tb]
\centering
\caption{The weights of the terminal Gorenstein weighted projective space in dimension $n$ that achieve the maximum value of $h$.}
\label{tab:max_h}
\begin{tabular}{ccl}
\toprule
$n$&$h$&\multicolumn{1}{c}{Weights}\\
\cmidrule(lr){1-1}\cmidrule(lr){2-2}\cmidrule(lr){3-3}
\oddrow $4$&$6$&$(1,1,1,1,2)$\\
\evnrow $5$&$12$&$(1,1,2,2,3,3)$\\
\evnrow &&$(1,1,1,2,3,4)$\\
\oddrow $6$&$60$&$(1,3,4,5,12,15,20)$\\
\evnrow $7$&$140$&$(1,2,5,14,20,28,35,35)$\\
\evnrow &&$(1,4,7,10,20,28,35,35)$\\
\oddrow $8$&$1260$&$(1,4,20,45,63,140,252,315,420)$\\
\oddrow &&$(1,9,20,28,35,180,252,315,420)$\\
\evnrow $9$&$12012$&$(12,13,21,44,273,924,1716,2002,3003,4004)$\\
\oddrow $10$&$427812$&$(1,2,462,924,9723,15279,19446,61116,71302,106953,142604)$\\
\bottomrule
\end{tabular}
\end{table}

\bibliographystyle{amsalpha}
\newcommand{\etalchar}[1]{$^{#1}$}
\providecommand{\bysame}{\leavevmode\hbox to3em{\hrulefill}\thinspace}
\providecommand{\MR}{\relax\ifhmode\unskip\space\fi MR }
\providecommand{\MRhref}[2]{%
  \href{http://www.ams.org/mathscinet-getitem?mr=#1}{#2}
}
\providecommand{\href}[2]{#2}

\end{document}